\documentclass[11pt,reqno]{amsart}
\usepackage{amsmath,amsthm,amssymb}
\usepackage{xcolor}

\newtheorem{theorem}{Theorem}

\newtheorem{proposition}[theorem]{Proposition}

\theoremstyle{definition}

\newtheorem{definition}[theorem]{Definition}

\newtheorem{remark}[theorem]{Remark}

\newtheorem{example}[theorem]{Example}

\begin{document}

\title{On the Chern-Ricci form of a twisted almost K\"{a}hler structure}
\author{David N. Pham and Fei Ye}

\address{Department of Mathematics $\&$ Computer Science, QCC CUNY, Bayside, NY 11364}
\curraddr{}
\email{dnpham@qcc.cuny.edu}
\email{fye@qcc.cuny.edu}



\dedicatory{}

\begin{abstract}
Let $(M,g,J,\omega)$ be an almost K\"{a}hler manifold.  For any smooth function $f$ on $M$, one can associate an automorphism $\psi\in \mbox{Aut}(TM)$ for which the K\"{a}hler form is invariant.  Using $\psi$, one can ``twist" the metric $g$ and almost complex structure $J$ to obtain a new almost K\"{a}hler structure $(g^\psi,J^\psi,\omega)$ on $M$.  Let $\widetilde{D}$ denote the Chern connection of $(g^\psi,J^\psi,\omega)$ and let $K^{-1}$ denote the anti-canonical bundle of $(TM,J^\psi)$.  In the current paper, we give an explicit formula for the local connection 1-form $\alpha$ associated to the pair $(K^{-1},\widetilde{D})$.  The Chern-Ricci form of $(g^\psi,J^\psi,\omega)$ is then $\rho_{\widetilde{D}}=-d\alpha$.  We note that under certain conditions the aforementioned formula assumes a simpler form  when applied to the calculation of $\alpha$.  We illustrate this with some examples.
\end{abstract}

\date{}

\maketitle

\section{Introduction}
\noindent Let $(M,g,J,\omega)$ be an almost K\"{a}hler manifold where 
\begin{itemize}
    \item[$\bullet$] $g$ is the metric,
    \item[$\bullet$] $J$ is the almost complex structure, and
    \item[$\bullet$] $\omega:=g(J\cdot,\cdot)$ is the K\"{a}hler form.
\end{itemize}
In other words, $(g,J,\omega)$ is an almost Hermitian structure where $\omega$ is a symplectic form.  One can show (see e.g. \cite{ApostolovDraghici2005}) that the symplectic condition implies that 
\begin{equation}
    \label{eqAKNJ}
    g((\nabla^g_ZJ)X,Y)=-\frac{1}{2}g(N_J(X,Y),JZ),
\end{equation}
where $\nabla^g$ denotes the Levi-Civita conneciton of $g$ and $N_J$ is the Nijenhuis tensor which we define with the following convention
$$
N_J(X,Y):=J[JX,Y]+J[X,JY]+[X,Y]-[JX,JY].
$$
Using equation (\ref{eqAKNJ}), one can further show that 
\begin{equation}
\label{eqJNabla}
\nabla^g_{JX}J=-J\nabla^g_{X}J.
\end{equation}
The Chern connection associated to $(g,J,\omega)$ is the unique connection $D$ on $TM$ for which $Dg=0$, $DJ=0$, and whose torsion $T^D$ has vanishing $(1,1)$ part with respect to $J$. Explicitly, $D$ is given by the following formula:
\begin{equation}
    \label{eqChernConnection}
    D_XY=\nabla^g_XY+\frac{1}{2}(\nabla^g_XJ)(JY).
\end{equation}
From this, it follows that the torsion is given by
$$
T^D = -\frac{1}{4}N_J.
$$
Let $T_JM$ be the complex vector bundle whose underlying real vector bundle is $TM$ and where scalar multiplication by $\sqrt{-1}$ is given by 
$$
\sqrt{-1}X:=JX.
$$
As $DJ=0$, we regard $D$ as a connection on the complex vector bundle $T_JM$. Note that if $\dim M=2n$, then $T_JM$ is a complex vector bundle of rank $n$.  Let 
$$
\widehat{g}:=g-\sqrt{-1}\omega.
$$
$\widehat{g}$ then defines a natural hermitian metric on $T_JM$, that is, 
$$
\widehat{g}(JX,Y)=\sqrt{-1}\widehat{g}(X,Y),\hspace*{0.2in}\widehat{g}(Y,X)=\overline{\widehat{g}(X,Y)}.
$$
Since $g$ and $J$ are both parallel with respect to $D$, it follows that $D\widehat{g}=0$.  The curvature endomorphism of $D$ is defined by
$$
R^D(X,Y):=[D_X,D_Y]-D_{[X,Y]}.
$$
The associated Chern-Ricci form is then
\begin{equation}
    \label{eqChernRicciForm}
    \rho_D(X,Y):=\mathrm{Tr}(R^D(X,Y)\circ J)=\sqrt{-1}\mathrm{Tr}(R^D(X,Y)),
\end{equation}
where the second equality follows from the fact that $J=\sqrt{-1}\mathrm{id}_{T_JM}$ under the identification $\sqrt{-1}X:=JX$.  In this paper, ``$\mathrm{Tr}$" will always mean the trace of a complex endomorphism. As a side remark, we recall that when the almost complex structure is integrable (i.e. the manifold is K\"{a}hler), the Chern-Ricci form coincides with the definition of the Ricci form. 

Let 
$$
e_1,\dots, e_n,~e_{n+1},\dots, e_{2n}
$$
be a $g$-orthonormal frame where $e_{n+i}:=Je_i$ for $i\le n$.  Then $e_1,\dots, e_n$ is a unitary frame on $T_JM$ and we may write 
$$
D_Xe_j=\sum_{i=1}^nA_{ij}(X)e_i
$$
where $A_{ij}$ are complex valued 1-forms whose domain is the same as the frame $\{e_i\}_{i=1}^n$. 
Taken together, $A:=(A_{ij})$ is the local connection 1-form for $D$ with respect to $\{e_i\}_{i=1}^n$.  

The matrix representation of $R^D$ with respect to $\{e_i\}_{i=1}^n$ is then the $n\times n$ matrix of 2-forms given by
$$
[R^D] = dA+A\wedge A
$$
where 
$$
(A\wedge A)(X,Y):=[A(X),A(Y)].
$$
It follows from this and the cyclic properties of the trace that 
\begin{equation}
\label{eqChernRicciTrace}
\rho_D=\sqrt{-1}~d\mathrm{Tr}(A)=\sqrt{-1}\sum_{i=1}^ndA_{ii}.
\end{equation}
Let $K^{-1}:=\wedge^n_{\mathbb{C}} T_JM$ be the \textit{anti-canonical bundle}.  Then $\xi:=e_1\wedge \cdots \wedge e_n$ is a local section of $K^{-1}$. Applying $D$ to $K^{-1}$ one has
$$
D_X\xi=\mathrm{Tr}(A)(X)\cdot \xi.
$$
Hence, $\mathrm{Tr}(A)$ is the local connection 1-form of the pair $(K^{-1},D)$. From this, it follows that the curvature of $(K^{-1},D)$ is 
$$
d\mathrm{Tr}(A)=-\sqrt{-1}\rho_D.
$$
Since $\widehat{g}(e_i,e_j)=\delta_{ij}$ and $D\widehat{g}=0$, it follows that
$$
A_{ji}=-\overline{A_{ij}}.
$$
In particular, the diagonal elements $A_{ii}$ are imaginary valued which, in turn, imply that the Chern-Ricci form $\rho_D$ is real valued.  As a consequence of this, equation (\ref{eqChernRicciTrace}) can be rewritten as
\begin{equation}
    \rho_D=-d\operatorname{Im}\mathrm{Tr}(A).
\end{equation}
\noindent For convenience, we make the following definition:
\begin{definition}
    Let $(M,g,J,\omega)$ be an almost K\"{a}hler manifold.  An automorphism $\psi\in \mathrm{Aut}(TM)$ is called a \textit{twist map} if it satisfies the following conditions:
    \begin{itemize}
        \item[(i)] $\psi$ is $g$-symmetric, that is, $\psi$ is self-adjoint with respect to $g$. 
        \item[(ii)] $J\psi =\psi^{-1}J$.
    \end{itemize}
\end{definition}
On an almost K\"{a}hler manifold $M$, twist maps exist in abundance.  Any smooth function $f\in C^\infty(M)$ gives rise to a twist map by the following construction\footnote{The authors wish to thank Professor Mehdi Lejmi for bringing this construction to their attention and for helpful discussions that motivated this work.} (see Theorem 4.7 of \cite{Blair2010}).  For $f\in C^\infty(M)$, let $X_f\in \mathfrak{X}(M)$ be the associated Hamiltonian vector field, that is,
$$
\omega(X_f,\cdot)=-df.
$$
Let $\widetilde{\psi}:=\mathcal{L}_{X_f}J$ where $\mathcal{L}_{X_f}$ is the Lie derivative with respect to $X_f$.  It follows immediately from this that 
$$
J\widetilde{\psi}=-\widetilde{\psi}J.
$$
One can show that $\widetilde{\psi}$ is $g$-symmetric if and only if $\beta_f(Y,Z)=\beta_f(Z,Y)$ where
$$
\beta_f(Y,Z):=g((\nabla^g_{Y}J)X_f,Z)-g(J(\nabla^g_{JY}J)X_f,Z)
$$
However, by equation (\ref{eqJNabla}), it follows that $\beta_f\equiv 0$.  Hence, $\widetilde{\psi}$ is $g$-symmetric.  The $g$-symmetry of $\widetilde{\psi}$ and its anticommutativity with $J$ implies that 
$$
\psi:=\mathrm{exp}(\widetilde{\psi})
$$ 
is a twist map. Using a twist map $\psi$, one obtains a new almost K\"{a}hler structure $(g^\psi,J^\psi,\omega)$ by defining
\begin{equation}
\label{eqPsiTwist}
g^\psi:=g(\psi^{-1}\cdot,\psi^{-1}\cdot),\hspace*{0.1in}J^\psi:=\psi\circ J\circ \psi^{-1}.
\end{equation}
Observe that the properties of $\psi$ imply that 
$$
\omega^\psi:=\omega(\psi^{-1}\cdot,\psi^{-1}\cdot)=\omega.
$$
Hence, a twist map $\psi$ deforms the metric and almost complex structure while leaving the K\"{a}hler form unchanged.  We will call the new almost K\"{a}hler structure $(g^\psi,J^\psi,\omega)$ a $\psi$-twist of $(g,J,\omega)$.

For convenience, let $(h,I,\omega):=(g^\psi,J^\psi,\omega)$.  Also, let $\tilde{D}$ be the Chern connection associated to $(M,h,I,\omega)$ and let $e_1,\dots, e_n$ be a local $\hat{g}$-unitary frame of $T_JM$. Define $f_1,\dots, f_n$ to be the $\hat{h}$-unitary frame on $T_IM$ given by $f_i:=\psi e_i$ and let $\widetilde{A}$ be the local connection 1-form of $\tilde{D}$ with respect to $f_1,\dots, f_n$.  In the current paper, we obtain the following formula for $\operatorname{Im}\mathrm{Tr}(\widetilde{A})$:
\begin{theorem}
    \label{thmMain}
    Let $(M,g,J,\omega)$ be an almost K\"{a}hler manifold and let $\psi$ be a twist map.  Let $h:=g^\psi$ and $I:=J^\psi$ and let $\widetilde{D}$ be the Chern connection of $(h,I,\omega)$.  Let $e_1,\dots, e_n$ be a $\widehat{g}$-unitary frame of $T_JM$ and let
    \begin{align}
        \nonumber
        \eta_P(X)&:=-\frac{1}{2}\sum_{i=1}^{2n}g(J\psi^{-1}[X,\psi e_i],e_i)\\
        \label{eqEtaP}
        &-\frac{1}{2}\sum_{i=1}^{2n}g(J\psi\nabla^g_{\psi e_i}(\psi^{-2}X),e_i),
    \end{align}
    where $e_{n+i}:=Je_i$ for $i\le n$.  Let $\widetilde{A}$ be the local connection 1-form of $\widetilde{D}$ associated to the $\widehat{h}$-unitary frame $f_1,\dots, f_n$ on $T_IM$ where $f_i:=\psi e_i$.  Then $\operatorname{Im}\mathrm{Tr}(\widetilde{A})=\eta_P$ and $\rho_{\widetilde{D}}=-d\eta_P$. 
    
    In addition, if 
    \begin{itemize}
        \item[(i)] $(\psi e_i)g(e_j,J\psi e_i)=0$ for all $i,j=1,\dots, 2n$, and
        \item[(ii)] $[\psi e_i,\psi J e_i]=0$ for $i\le n$,
    \end{itemize}
    then
    $$
        \eta_P(e_j)=-\frac{1}{2}\sum_{i=1}^{2n}g(J\psi^{-1}[e_j,\psi e_i],e_i)
    $$
    for $j=1,\dots, 2n$.
\end{theorem}
\begin{remark}
We note that P. Gauduchon derived an alternate formula (see equation (9.5.9) of \cite{GauduchonCalabi}) for $\operatorname{Im}\mathrm{Tr}(\widetilde{A})$.  As far as we can tell, the two formulas do not seem to be derivable from one another by any simple algebraic transformation.  We also note that Gauduchon's manuscript \cite{GauduchonCalabi}, which contains his formula, was never formally published and thus remains largely inaccessible to the broader mathematical community. 
\end{remark}
\noindent As a side remark, we note that with $I:=J^\psi$, the twist map $\psi: T_JM\stackrel{\sim}{\rightarrow} T_IM$ is an isomorphism of complex vector bundles which, in turn, implies that
$$
2\pi c_1(M)=[\rho_D]=[\rho_{\tilde{D}}],
$$
where $c_1(M)$ is the first Chern class of $T_JM$.  Hence, the Chern-Ricci form $\rho_{\tilde{D}}$ associated to the twisted almost K\"{a}hler structure $(g^\psi,J^\psi,\omega)$ remains in the same cohomology class as the Chern-Ricci form $\rho_D$ of the untwisted almost K\"{a}hler structure $(g,J,\omega)$. 

The Chern-Ricci form is (of course) the central object in the Chern-Ricci flow \cite{Gill2011,TosattiWeinkove2015,TosattiWeinkove2013,ShermanWeinkove2013} and one hopes to use the latter as a tool to study complex manifolds in a manner which is analogous to how Hamilton and Perelman used the Ricci flow to study 3-manifolds  \cite{Hamilton1982,Perelman2002,Perelman2003a,Perelman2003b}.  More recently, Zheng extended the idea of Chern-Ricci flow to the almost complex case \cite{Zheng2017} (which is the domain of interest in the current paper).  As a practical matter, it is certainly useful and beneficial to derive precise formulas for the Chern-Ricci form under various deformations given the importance and potential benefit of the Chern-Ricci flow.  The formula in Theorem \ref{thmMain} provides a direct method for computing the Chern-Ricci form in the present framework without explicitly determining the Chern connection, which can be computationally involved.  The rest of the paper is organized as follows.  In Section \ref{SecMainTheorem}, we give a proof of Theorem \ref{thmMain}.  In Section \ref{SecExamples}, we present some examples where we calculate $\operatorname{Im}\mathrm{Tr}(\widetilde{A})$ using the formula $\eta_P$ in Theorem \ref{thmMain}. The examples considered satisfy conditions (i) and (ii) of Theorem \ref{thmMain} which further simplifies our calculation of $\eta_P$.  The paper concludes in Section \ref{secConclusion} with some closing remarks and directions for future work.

\section{Proof of Theorem \ref{thmMain}}
\label{SecMainTheorem}
\noindent Let $e_1,\dots, e_n$ be a fixed $\widehat{g}$-unitary frame on $T_JM$.  Then it follows that 
$$
e_1,\dots, e_n,~e_{n+1},\dots,e_{2n}
$$
is a $g$-orthonormal frame on $TM$ where $e_{n+i}:=Je_i$ for $i\le n$.  Let $f_i:=\psi e_i$ for $i\le n$ and let $f_{n+i}:=If_i=\psi e_{n+i}$ for $i\le n$.  Then
$$
f_1,\dots, f_n,~f_{n+1},\dots, f_{2n}
$$
is an $h$-orthonormal frame on $TM$. In particular, $f_1,\dots, f_n$ is an $\widehat{h}$-unitary frame on $T_IM$.  Let $\widetilde{A}$ be the connection 1-form of $\widetilde{D}$ associated to $f_1,\dots, f_n$.  From this, we have
\begin{align*}
    \mathrm{Tr}(\widetilde{A})(X)&=\sum_{i=1}^n\widehat{h}(\widetilde{D}_Xf_i,f_i)\\
    &=-\sqrt{-1}\sum_{i=1}^n\omega(\widetilde{D}_Xf_i,f_i),
\end{align*}
where we use the fact that $\mathrm{Tr}(\widetilde{A})$ is purely imaginary since $f_1,\dots, f_n$ is an $\widehat{h}$-unitary frame.  From this, we have
\begin{equation}
    \label{eqTraceAtilde1}
    -\operatorname{Im}\mathrm{Tr}(\widetilde{A})(X)=\sum_{i=1}^n\omega(\widetilde{D}_Xf_i,f_i).
\end{equation}
\noindent Define
\begin{equation}
    \label{eqDefE}
    E(X,Y):=\nabla^h_XY-\nabla^g_XY,
\end{equation}
where we recall that $\nabla^g$ and $\nabla^h$ are the Levi-Civita connections of $g$ and $h$ respectively.  Using the almost K\"{a}hler identity (\ref{eqJNabla}), we can rewrite the Chern connection of $(h,I,\omega)$ as
\begin{align}
    \label{eqDtilde1}
    \widetilde{D}_XY &= \nabla^h_XY+\frac{1}{2}(\nabla^h_{IX}I)Y.
\end{align}
Using (\ref{eqDefE}), the second term in (\ref{eqDtilde1}) expands as 
\begin{equation}
    \label{eqNablaHIXI}
    \frac{1}{2}(\nabla^h_{IX}I)Y= \frac{1}{2}(\nabla^g_{IX}I)Y+ \frac{1}{2}E(IX,IY)- \frac{1}{2}IE(IX,Y).
\end{equation}
Rewriting (\ref{eqTraceAtilde1}) with the help of (\ref{eqDtilde1}) and (\ref{eqNablaHIXI}) gives
\begin{align}
    \nonumber
    -\operatorname{Im}\mathrm{Tr}(\widetilde{A})(X)&=\sum_{i=1}^n\omega(\nabla^h_Xf_i,f_i)+\frac{1}{2}\sum_{i=1}^n\omega((\nabla^h_{IX}I)f_i,f_i)\\
    \nonumber
    &=\sum_{i=1}^n\omega(\nabla^g_{X}f_i,f_i)+\sum_{i=1}^n\omega(E(X,f_i),f_i)\\
    \nonumber
    &+\frac{1}{2}\sum_{i=1}^n\omega((\nabla^g_{IX}I)f_i,f_i)+\frac{1}{2}\sum_{i=1}^n\omega(E(IX,If_i),f_i)\\
    \label{eqPhaseA}
    &-\frac{1}{2}\sum_{i=1}^n\omega(IE(IX,f_i),f_i).
\end{align}
Note that $If_1,\dots, If_n$ is also an $\hat{h}$-unitary frame on $T_IM$.  Moreover, the local connection 1-form of $\widetilde{D}$ associated to the aforementioned frame is also $\widetilde{A}$.  Indeed, one has
\begin{align*}
    \widetilde{D}_X(If_j)&=I\widetilde{D}_Xf_j=I\sum_{i=1}^n\widetilde{A}_{ij}(X)f_i=\sum_{i=1}^n\widetilde{A}_{ij}(X)(If_i).
\end{align*}
Hence, (\ref{eqPhaseA}) remains valid under the substitution $f_i\rightarrow If_i$:
\begin{align}
    \nonumber
    -\operatorname{Im}\mathrm{Tr}(\widetilde{A})(X)&=\sum_{i=1}^n\omega(\nabla^g_{X}(If_i),If_i)+\sum_{i=1}^n\omega(E(X,If_i),If_i)\\
    \nonumber
    &+\frac{1}{2}\sum_{i=1}^n\omega((\nabla^g_{IX}I)(If_i),If_i)-\frac{1}{2}\sum_{i=1}^n\omega(E(IX,f_i),If_i)\\
    \nonumber
    &-\frac{1}{2}\sum_{i=1}^n\omega(IE(IX,If_i),If_i)\\
    \nonumber
    &=\sum_{i=1}^n\omega(\nabla^g_{X}(If_i),If_i)+\sum_{i=1}^n\omega(E(X,If_i),If_i)\\
    \nonumber
    &-\frac{1}{2}\sum_{i=1}^n\omega((\nabla^g_{IX}I)f_i,f_i)+\frac{1}{2}\sum_{i=1}^n\omega(IE(IX,f_i),f_i)\\
    \label{eqPhaseB}
    &-\frac{1}{2}\sum_{i=1}^n\omega(E(IX,If_i),f_i)
\end{align}
Taking the sum of (\ref{eqPhaseA}) and (\ref{eqPhaseB}) gives
\begin{align}
    \nonumber
    -2\operatorname{Im}\mathrm{Tr}(\widetilde{A})(X)=\sum_{i=1}^n\omega(\nabla^g_{X}f_i,f_i)+\sum_{i=1}^n\omega(\nabla^g_{X}(If_i),If_i)\\
    \label{eqPhaseC}
    +\sum_{i=1}^n\omega(E(X,f_i),f_i)+\sum_{i=1}^n\omega(E(X,If_i),If_i).
\end{align}
For $i\le n$, we now substitute $f_i=\psi e_i$ and $If_i=\psi (Je_i)=\psi e_{n+i}$ into (\ref{eqPhaseC}) and rewrite all the $\omega$'s in terms of $g$.  This gives
\begin{align}
\nonumber
-2\operatorname{Im}\mathrm{Tr}(\widetilde{A})(X)&=\sum_{i=1}^{2n}g(\psi J\nabla^g_X(\psi e_i),e_i)+\sum_{i=1}^{2n}g(\psi J E(X,\psi e_i),e_i)\\
\label{eqPhaseD}
&=\sum_{i=1}^{2n}g(J\psi^{-1}\nabla^g_X(\psi e_i),e_i)+\sum_{i=1}^{2n}g(J\psi^{-1} E(X,\psi e_i),e_i),
\end{align}
where we have made use of the fact that $\psi$ is $g$-symmetric and $\psi J=J\psi^{-1}$.

For convenience, we will often use the following notation:
\begin{equation}
    \gamma_X:=\nabla^g_{X}\psi^{-1}
\end{equation}
Note that since $\psi$ (and hence $\psi^{-1}$) is $g$-symmetric, one can easily show that $\gamma_X$ is $g$-symmetric as well.  To obtain the desired formula for $\operatorname{Im}\mathrm{Tr}(\widetilde{A})$, we will need to compute $E$ explicitly.
\begin{proposition}
\label{propE}
For $X,Y,Z\in \mathfrak{X}(M)$, let $Q(X,Y)\in \mathfrak{X}(M)$ be the unique vector field given by
\begin{equation}
    \label{eqQ}
g(Q(X,Y),Z)=g(\psi^{-1}\gamma_ZX,Y)+g(\psi^{-1}\gamma_ZY,X).
\end{equation}
Then
\begin{align}
    \nonumber
    E(X,Y)&=\frac{1}{2}\psi\left[\gamma_XY+\gamma_YX+\psi \gamma_X\psi^{-1}Y+\psi \gamma_Y\psi^{-1}X-\psi Q(X,Y)\right].
\end{align}
\end{proposition}
\begin{proof}
    By Proposition 2.1 of \cite{Pham2026}, we have
    \begin{align*}
        2h(\nabla^h_XY,Z)&=2h(\nabla^g_XY,Z)+h(\psi\gamma_XY,Z)+h(\psi\gamma_YX,Z)\\
        &+h(\psi \gamma_XZ,Y)-h(\psi\gamma_ZX,Y)+h(\psi \gamma_YZ,X)\\
        &-h(\psi\gamma_ZY,X).
    \end{align*}
    Applying the definition of $h$ and using the $g$-symmetry of $\psi^{-1}$ and $\gamma_X$, the above expression can be rewritten as
    \begin{align*}
        2g(\psi^{-2}\nabla^h_XY,Z)&=2g(\psi^{-2}\nabla^g_XY,Z)+g(\psi^{-1}\gamma_XY,Z)+g(\psi^{-1}\gamma_YX,Z)\\
        &+g( Z,\gamma_X\psi^{-1}Y)-g(\psi^{-1}\gamma_ZX,Y)+g( Z,\gamma_Y\psi^{-1}X)\\
        &-g(\psi^{-1}\gamma_ZY,X).
    \end{align*}
    Applying the definition of $Q$, we have
    \begin{align*}
        2g(\psi^{-2}\nabla^h_XY,Z)&=2g(\psi^{-2}\nabla^g_XY,Z)+g(\psi^{-1}\gamma_XY,Z)+g(\psi^{-1}\gamma_YX,Z)\\
        &+g(\gamma_X\psi^{-1}Y,Z)+g(\gamma_Y\psi^{-1}X,Z)\\
        &-g(Q(X,Y),Z).
    \end{align*}
    The nondegeneracy of $g$ implies that 
    \begin{align*}
        2\psi^{-2}\nabla^h_XY&=2\psi^{-2}\nabla^g_XY+\psi^{-1}\gamma_XY+\psi^{-1}\gamma_YX\\
        &+\gamma_X\psi^{-1}Y+\gamma_Y\psi^{-1}X-Q(X,Y).
    \end{align*}
    Lastly, applying $\frac{1}{2}\psi^2$ to both sides gives
    $$
    \nabla^h_XY=\nabla^g_XY+E(X,Y).
    $$
\end{proof}
\noindent We now decompose the second term in (\ref{eqPhaseD}) as follows:
\begin{align}
    \nonumber
    \sum_{i=1}^{2n}g(J\psi^{-1}E(X,\psi e_i),e_i)&=\sum_{i=1}^ng(J\psi^{-1}E(X,\psi e_i),e_i)\\
    \label{eqPhaseE}
    &+\sum_{i=1}^ng(J\psi^{-1}E(X,\psi Je_i),Je_i)
\end{align}
where we recall that $e_{n+i}:=Je_i$ for $i\le n$. Using Proposition \ref{propE}, the first term in (\ref{eqPhaseE}) expands as 
\begin{align}
    \nonumber
    \sum_{i=1}^n&g(J\psi^{-1}E(X,\psi e_i),e_i)=-\sum_{i=1}^ng(\psi^{-1}E(X,\psi e_i),Je_i)\\
    \nonumber
    &=-\frac{1}{2}\sum_{i=1}^ng(\gamma_X(\psi e_i),Je_i)-\frac{1}{2}\sum_{i=1}^ng(\gamma_{\psi e_i}X,Je_i)\\
    \nonumber
    &-\frac{1}{2}\sum_{i=1}^ng(\psi \gamma_Xe_i,Je_i)-\frac{1}{2}\sum_{i=1}^ng(\psi \gamma_{\psi e_i}(\psi^{-1}X),Je_i)\\
    \nonumber
    &+\frac{1}{2}\sum_{i=1}^ng( Q(X,\psi e_i),\psi Je_i)\\
    \nonumber
    &=-\frac{1}{2}\sum_{i=1}^ng(\psi\gamma_X(Je_i),e_i)
    -\frac{1}{2}\sum_{i=1}^ng(\gamma_{\psi e_i}X,Je_i)\\
    \nonumber
    &-\frac{1}{2}\sum_{i=1}^ng(\gamma_X(\psi Je_i),e_i)
    -\frac{1}{2}\sum_{i=1}^ng(\psi^{-1}\gamma_{\psi e_i} (\psi Je_i),X)\\
    \label{eqPhaseF}
    &+\frac{1}{2}\sum_{i=1}^ng(\psi^{-1}\gamma_{\psi Je_i}X,\psi e_i)+\frac{1}{2}\sum_{i=1}^ng(\psi^{-1}\gamma_{\psi Je_i}(\psi e_i),X).
\end{align}
The last two terms coming from the $Q$-term in (\ref{eqPhaseF}) can be rewritten using $g$-symmetry.  The result is then 
\begin{align}
    \nonumber
    \sum_{i=1}^n&g(J\psi^{-1}E(X,\psi e_i),e_i)=-\frac{1}{2}\sum_{i=1}^ng(\psi\gamma_X(Je_i),e_i)
    -\frac{1}{2}\sum_{i=1}^ng(\gamma_{\psi e_i}X,Je_i)\\
    \nonumber
    &-\frac{1}{2}\sum_{i=1}^ng(\gamma_X(\psi Je_i),e_i)
    -\frac{1}{2}\sum_{i=1}^ng(\psi^{-1}\gamma_{\psi e_i} (\psi Je_i),X)\\
    \label{eqPhaseG}
    &+\frac{1}{2}\sum_{i=1}^ng(\gamma_{\psi Je_i}X, e_i)+\frac{1}{2}\sum_{i=1}^ng(\psi\gamma_{\psi Je_i}(\psi^{-1}X),e_i).
\end{align}
\noindent In a completely similar way, we expand the second term in (\ref{eqPhaseE}) using Proposition \ref{propE}.
\begin{align}
    \nonumber
    \sum_{i=1}^n&g(J\psi^{-1}E(X,\psi Je_i),Je_i)=\sum_{i=1}^ng(\psi^{-1}E(X,\psi Je_i),e_i)\\
    \nonumber
    &=\frac{1}{2}\sum_{i=1}^ng(\gamma_X(\psi Je_i),e_i)+\frac{1}{2}\sum_{i=1}^ng(\gamma_{\psi Je_i}X,e_i)\\
    \nonumber
    &+\frac{1}{2}\sum_{i=1}^ng(\psi\gamma_{X}(Je_i),e_i)+\frac{1}{2}\sum_{i=1}^ng(\psi \gamma_{\psi Je_i}(\psi^{-1}X),e_i)\\
    \label{eqPhaseH}
    &-\frac{1}{2}\sum_{i=1}^ng(\gamma_{\psi e_i}X,J e_i)-\frac{1}{2}\sum_{i=1}^ng(\psi^{-1}\gamma_{\psi e_i}(\psi J e_i),X).
\end{align}
Substituting (\ref{eqPhaseG}) and (\ref{eqPhaseH}) into (\ref{eqPhaseE}) gives
\begin{align}
    \nonumber
    \sum_{i=1}^{2n}&g(J\psi^{-1}E(X,\psi e_i),e_i)=-\sum_{i=1}^ng(\gamma_{\psi e_i}X,Je_i)-\sum_{i=1}^ng(\psi^{-1}\gamma_{\psi e_i}(\psi Je_i),X)\\
    \nonumber
    &+\sum_{i=1}^ng(\gamma_{\psi Je_i}X,e_i)+\sum_{i=1}^ng(\psi \gamma_{\psi Je_i}(\psi^{-1}X),e_i)\\
    \nonumber
    &=\sum_{i=1}^ng(J\gamma_{\psi e_i}X,e_i)+\sum_{i=1}^ng(J\psi\gamma_{\psi e_i}(\psi^{-1}X),e_i)\\
    \label{eqPhaseI}
    &+\sum_{i=1}^ng(J\gamma_{\psi Je_i}X,Je_i)+\sum_{i=1}^ng(J\psi \gamma_{\psi Je_i}(\psi^{-1}X),Je_i)
\end{align}
Using $e_{n+i}:=Je_i$ for $i\le n$, (\ref{eqPhaseI}) can be rewritten as
\begin{align}
\label{eqPhaseJ}
\sum_{i=1}^{2n}g(J\psi^{-1}E(X,\psi e_i),e_i)=\sum_{i=1}^{2n}g(J\gamma_{\psi e_i}X,e_i)+\sum_{i=1}^{2n}g(J\psi \gamma_{\psi e_i}(\psi^{-1}X),e_i).
\end{align}
Substituting (\ref{eqPhaseJ}) into (\ref{eqPhaseD}) gives
\begin{align}
    \nonumber
    -2\operatorname{Im}&\mathrm{Tr}(\widetilde{A})(X)=\sum_{i=1}^{2n}g(J\psi^{-1}\nabla^g_X(\psi e_i),e_i)\\
    \label{eqPhaseK}
    &+\sum_{i=1}^{2n}g(J(\nabla^g_{\psi e_i}\psi^{-1})X,e_i)+\sum_{i=1}^{2n}g(J\psi (\nabla^g_{\psi e_i}\psi^{-1})(\psi^{-1}X),e_i),
\end{align}
where we have recalled that $\gamma_X:=\nabla^g_X\psi^{-1}$. Define 
$$
\Xi_i:=J\psi^{-1}\nabla^g_X(\psi e_i)+J(\nabla^g_{\psi e_i}\psi^{-1})X+J\psi (\nabla^g_{\psi e_i}\psi^{-1})(\psi^{-1}X).
$$
Expanding $\Xi_i$ gives
\begin{align}
    \nonumber
    \Xi_i&=J\psi^{-1}\nabla^g_X(\psi e_i)+J\nabla^g_{\psi e_i}(\psi^{-1}X)\\
    \nonumber
    &-J\psi^{-1}\nabla^g_{\psi e_i}X+J\psi\nabla^g_{\psi e_i}(\psi^{-2}X)\\
    \nonumber
    &-J\nabla^g_{\psi e_i}(\psi^{-1}X)\\
    \label{eqPhaseL}
    &=J\psi^{-1}[X,\psi e_i]+J\psi\nabla^g_{\psi e_i}(\psi^{-2}X).
\end{align}
Applying (\ref{eqPhaseL}) to (\ref{eqPhaseK}) gives
\begin{align*}
    \operatorname{Im}\mathrm{Tr}(\widetilde{A})(X)&=-\frac{1}{2}\sum_{i=1}^{2n}g(J\psi^{-1}[X,\psi e_i],e_i)-\frac{1}{2}\sum_{i=1}^{2n}g(J\psi \nabla^g_{\psi e_i}(\psi^{-2}X),e_i)\\
    &=\eta_P(X).
\end{align*}

\noindent For the last statement of Theorem \ref{thmMain}, let us now assume that conditions (i) and (ii) of Theorem \ref{thmMain} are satisfied.  Let
$$
S_1(X):=-\frac{1}{2}\sum_{i=1}^{2n}g(J\psi^{-1}[X,\psi e_i],e_i)
$$
and 
$$
S_2(X):=-\frac{1}{2}\sum_{i=1}^{2n}g(J\psi \nabla^g_{\psi e_i}(\psi^{-2}X),e_i).
$$
Then $\eta_P=S_1+S_2$. Using condition (i), we observe that
\begin{align*}
    (\psi e_i)g(\psi^{-2}e_j,\psi Je_i)&=(\psi e_i)g(e_j,\psi^{-1}Je_i)\\
    &=(\psi e_i)g(e_j,J\psi e_i)\\
    &=0.
\end{align*}
The above identity together with the metric compatibility of $\nabla^g$ now implies 
\begin{align*}
    S_2(e_j)&=-\frac{1}{2}\sum_{i=1}^{2n}g(J\psi \nabla^g_{\psi e_i}(\psi^{-2}e_j),e_i)\\
    &=\frac{1}{2}\sum_{i=1}^{2n}g( \nabla^g_{\psi e_i}(\psi^{-2}e_j),\psi Je_i)\\
    &=-\frac{1}{2}\sum_{i=1}^{2n}g(\psi^{-2}e_j,\nabla^g_{\psi e_i}(\psi Je_i))\\
    &=-\frac{1}{2}\sum_{i=1}^{n}g(\psi^{-2}e_j,\nabla^g_{\psi e_i}(\psi Je_i))-\frac{1}{2}\sum_{i=n+1}^{2n}g(\psi^{-2}e_j,\nabla^g_{\psi e_i}(\psi Je_i))\\
    &=-\frac{1}{2}\sum_{i=1}^{n}g(\psi^{-2}e_j,\nabla^g_{\psi e_i}(\psi Je_i))+\frac{1}{2}\sum_{i=1}^{n}g(\psi^{-2}e_j,\nabla^g_{\psi Je_i}(\psi e_i))\\
\end{align*}
where we have used the fact that $e_{n+i}:=Je_i$ in the fifth equality.  Using the fact that $\nabla^g$ is torsion free gives
\begin{align*}
    S_2(e_j)&=-\frac{1}{2}\sum_{i=1}^{n}g(\psi^{-2}e_j,\nabla^g_{\psi e_i}(\psi Je_i))+\frac{1}{2}\sum_{i=1}^{n}g(\psi^{-2}e_j,\nabla^g_{\psi Je_i}(\psi e_i))\\
    &=-\frac{1}{2}\sum_{i=1}^{n}g(\psi^{-2}e_j,\nabla^g_{\psi e_i}(\psi Je_i))+\frac{1}{2}\sum_{i=1}^{n}g(\psi^{-2}e_j,\nabla^g_{\psi e_i}(\psi Je_i))\\
    &+\frac{1}{2}\sum_{i=1}^ng(\psi^{-2}e_j,[\psi J e_i,\psi e_i])\\
    &=\frac{1}{2}\sum_{i=1}^ng(\psi^{-2}e_j,[\psi J e_i,\psi e_i])\\
    &=0,
\end{align*}
where the last equality follows from condition (ii) of Theorem \ref{thmMain}.   Hence, $\eta_P(e_j)=S_1(e_j)$ under the assumptions of conditions (i) and (ii) of Theorem \ref{thmMain}.
\section{Examples}
\label{SecExamples}
In this section, we apply the new formula to a simple 4-dimensional (strictly) almost K\"{a}hler manifold. Let $G$ be the Lie group $H\times \mathbb{R}$ where $H$ is the $3$-dimensional Heisenberg Lie group.  Explicitly, $H$ is given by
$$
H=\left\{\begin{pmatrix}
    1 & x_1 & x_2\\
    0 & 1 & x_3\\
    0 & 0 & 1
\end{pmatrix}~|~x_1,x_2,x_3\in \mathbb{R}\right\}.
$$
$G$ has global coordinates $x_1,\dots, x_4$ where $x_4$ is the natural coordinate on $\mathbb{R}$.  Let $\mathfrak{h}:=\mathrm{Lie}(H)$ be the Lie algebra of left-invariant vector fields on $H$ which we identity with the tangent space of $H$ at the identity.  A basis for $\mathfrak{h}$ is 
$$
    \tilde{e}_1:=\begin{pmatrix}
        0 & 1 & 0\\
        0 & 0 & 0\\
        0 & 0 & 0
    \end{pmatrix},~\tilde{e}_2:=\begin{pmatrix}
        0 & 0 & 0\\
        0 & 0 & 1\\
        0 & 0 & 0
    \end{pmatrix},~\tilde{e}_3:=\begin{pmatrix}
        0 & 0 & 1\\
        0 & 0 & 0\\
        0 & 0 & 0
    \end{pmatrix}.
$$
The only nonzero bracket on $\mathfrak{h}$ is $[\tilde{e}_1,\tilde{e}_2]=\tilde{e}_3$. The Lie algebra of $G$ is $\mathfrak{g}=\mathfrak{h}\oplus \mathbb{R}$.  A basis for $\mathfrak{g}$ is
$$
e_j:=(\tilde{e}_j,0),~j=1,2,3,\hspace*{0.1in}e_4:=(0,1).
$$
Let $\partial_j:=\frac{\partial}{\partial x_j}$ for $j=1,\dots, 4$.  Expressing the left-invariant vector fields $\{e_j\}$ in terms of the coordinate tangent vector fields one has
$$
e_1=\partial_1,~e_2=x_1\partial_2+\partial_3,~e_3=\partial_2,~e_4=\partial_4.
$$
Let $\theta^1,\dots, \theta^4$ denote the left-invariant 1-forms on $\mathfrak{g}$ dual to $e_1,\dots, e_4$.  Then
$$
dx_1=\theta^1,~dx_2=x_1\theta^2+\theta^3,~dx_3=\theta^2,~dx_4=\theta^4.
$$

We consider a left-invaraint almost K\"{a}hler structure $(g,J,\omega)$ on $G$ which is defined as follows:
$$
g(e_i,e_j)=\delta_{ij},\hspace*{0.1in}\omega =\theta^1\wedge \theta^3+\theta^2\wedge \theta^4,
$$
$$
Je_1=e_3,~Je_2=e_4,~Je_3=-e_1,~Je_4=-e_2.
$$
We write the components of the Levi-Civita connection of $g$ as $\nabla^g_{e_i}e_j=\sum_{k=1}^4\Gamma_{ij}^ke_k$.  The nonzero components are then
$$
\Gamma^1_{23}=\frac{1}{2},~\Gamma^1_{32}=\frac{1}{2},~\Gamma^2_{13}=-\frac{1}{2},~\Gamma^2_{31}=-\frac{1}{2},~\Gamma^3_{12}=\frac{1}{2},~\Gamma^3_{21}=-\frac{1}{2}.
$$
Let $S_1$ and $S_2$ be defined as in Section \ref{SecMainTheorem}.  Then $\eta_P=S_1+S_2$. Note that while $\eta_P$ is a 1-form, $S_1$ and $S_2$, considered individually, are not $1$-forms.  For the examples considered below, we choose functions $f$ on $G$ for which 
$$
\widetilde{\psi}:=\mathcal{L}_{X_f}J
$$
is sufficiently simple, that is, the associated twist map $\psi:=\mathrm{exp}(\widetilde{\psi})$ can be computed explicitly.  Moreover, the functions $f$ have been chosen so that conditions (i) and (ii) of Theorem \ref{thmMain} are satisfied, which  will further simplify the calculation of $\operatorname{Im}\mathrm{Tr}(\widetilde{A})$.

\begin{example}
    \label{example1}
    Let $f=\lambda x_2$ where $\lambda \in \mathbb{R}$.  By direct calculation, one finds that 
    $$
    X_f=-\lambda e_1+\lambda x_1 e_4.
    $$
    $\widetilde{\psi}:=\mathcal{L}_{X_f}J$ is given by
    $$
     \widetilde{\psi}e_1=-\lambda e_2,~\widetilde{\psi}e_2=-\lambda e_1,~\widetilde{\psi}e_3=\lambda e_4,~\widetilde{\psi}e_4=\lambda e_3.
    $$
    $\psi:=\mbox{exp}(\widetilde{\psi})$ is then given by
    $$
    \psi e_1=ce_1-s e_2,\hspace*{0.1in}\psi e_2=-s e_1+c e_2,
    $$
    $$
    \psi e_3 = c e_3+s e_4,\hspace*{0.1in} \psi e_4 = s e_3+c e_4,
    $$
    where $c:=\cosh(\lambda)$ and $s:=\sinh(\lambda)$.  The inverse of $\psi$ is easily computed:
    $$
    \psi^{-1}e_1=ce_1+se_2,\hspace*{0.1in} \psi^{-1}e_2=se_1+ce_2
    $$
    $$
    \psi^{-1}e_3=ce_3-se_4,\hspace*{0.1in}\psi^{-1}e_4=-se_3+ce_4.
    $$
    Since $g$ and $\psi$ are left-invariant, it follows that 
    $$
     (\psi e_i)g(e_j,\psi e_i)=0
    $$
    for all $i,j$.  Moreover, let $U:=\mbox{span}\{e_1,e_2\}$ and $V:=\mbox{span}\{e_3,e_4\}$.  For this example, we have $\psi U\subset U$ and $\psi V\subset V$.  Given that $[\mathfrak{g},V]=0$ and $JU\subset V$, we have
    $$
     [\psi e_i, \psi Je_i] = 0
    $$
    for $i=1,2$.  By the second statement of Theorem \ref{thmMain}, $S_2(e_j)=0$ for all $j$.  Hence,  
    $$
    \eta_P(e_j) = S_1(e_j):=-\frac{1}{2}\sum_{i=1}^{4}g(J\psi^{-1}[e_j,\psi e_i],e_i).
    $$
    In addition, since $[\mathfrak{g},V]=0$, we immediately have 
    $$
        \eta_P(e_3)=\eta_P(e_4)=0.
    $$
    Moreover, since $\psi e_3,~\psi e_4\in V:=\mbox{span}\{e_3,e_4\}$, $\eta_P(e_j)$ further reduces to 
    $$
    \eta_P(e_j)=-\frac{1}{2}\sum_{i=1}^{2}g(J\psi^{-1}[e_j,\psi e_i],e_i).
    $$
    To determine $\eta_P(e_1)$, we compute 
    \begin{align*}
        J\psi^{-1}[e_1,\psi e_1]&=cse_1-s^2e_2\\
        J\psi^{-1}[e_1,\psi e_2]&=-c^2e_1+cse_2.
    \end{align*}
    From this, it follows that 
    $$
    \eta_P(e_1)=-\frac{1}{2}cs-\frac{1}{2}cs=-cs.
    $$
    To determine $\eta_P(e_2)$, we compute 
    \begin{align*}
        J\psi^{-1}[e_2,\psi e_1]&=c^2e_1-cse_2\\
        J\psi^{-1}[e_2,\psi e_2]&=-cse_1+s^2e_2.
    \end{align*}
    This implies 
    $$
    \eta_P(e_2)=-\frac{1}{2}c^2-\frac{1}{2}s^2=-\frac{1}{2}(c^2+s^2).
    $$
    Putting everything together, we conclude that 
    $$
    \eta_P = -cs\theta^1-\frac{1}{2}(c^2+s^2)\theta^2,
    $$
    where, again, $c:=\cosh(\lambda)$ and $s:=\sinh(\lambda)$.

    For comparison, we also calculate the connection 1-form $\widetilde{A}$ with respect to the twisted Chern connection $\widetilde{D}$ on $T_IG$ with respect to the frame $f_1:=\psi e_1$, $f_2:=\psi e_2$.  By a direct (and lengthy) calculation, the four components of $\widetilde{A}$ are given by
    \begin{align*}
        \widetilde{A}_{11}&=\sqrt{-1}\left(-\frac{cs}{2}\theta^1-\frac{c^2}{2}\theta^2 \right)\\
        \widetilde{A}_{21}&=-\frac{c^2+s^2}{4}\theta^3+\frac{cs}{2}\theta^4+\sqrt{-1}\left[\frac{c^2+s^2}{4}\theta^1+\frac{cs}{2}\theta^2 \right]\\
        \widetilde{A}_{12}&=\frac{c^2+s^2}{4}\theta^3-\frac{cs}{2}\theta^4+\sqrt{-1}\left[\frac{c^2+s^2}{4}\theta^1+\frac{cs}{2}\theta^2 \right]\\
        \widetilde{A}_{22}&=\sqrt{-1}\left[-\frac{cs}{2}\theta^1-\frac{s^2}{2}\theta^2\right].
    \end{align*}
    Hence, 
    $$
    \operatorname{Im}\mathrm{Tr}(\widetilde{A})=-cs\theta^1-\frac{c^2+s^2}{2}\theta^2=\eta_P.
    $$    
\end{example}

\newpage
\begin{example}
Let $f=x_4^2$.  From $\omega(X_f,\cdot)=-df$, we see that 
$$
X_f=-2x_4e_2.
$$
Computing $\widetilde{\psi}=\mathcal{L}_{X_f}J$, we obtain
$$
\widetilde{\psi} e_1  = 2x_4 e_1, \quad\quad
\widetilde{\psi} e_2 =  2 e_2, \quad\quad
\widetilde{\psi} e_3 =  -2x_4 e_3, \quad\quad
\widetilde{\psi} e_4 = - 2 e_4.
$$
Since $\widetilde{\psi}$ is diagonal, $\psi:=\exp(\widetilde{\psi})$ is especially easy to calculate in the current example:
$$
\psi e_1  = a e_1, \quad\quad
\psi e_2 =  b e_2, \quad\quad
\psi e_3 =  a^{-1} e_3, \quad\quad
\psi e_4 = b^{-1} e_4,
$$
where $a:=e^{2x_4}$ and $b:=e^2$.  By inspection, one verifies that conditions (i) and (ii) of Theorem \ref{thmMain} are satisfied here.  Hence, $\eta_P(e_j)=S_1(e_j)$ for $j=1,\dots, 4$.  By direct calculation, we have the following:
\begin{align*}
S_1(e_1)&=-\frac{1}{2}\sum_{i=1}^4g(J\psi^{-1}[e_1,\psi e_i],e_i)=0,\\
S_1(e_2)&=-\frac{1}{2}\sum_{i=1}^4g(J\psi^{-1}[e_2,\psi e_i],e_i)=-\frac{1}{2}e^{4x_4},\\
S_1(e_3)&=-\frac{1}{2}\sum_{i=1}^4g(J\psi^{-1}[e_3,\psi e_i],e_i)=0,\\
S_1(e_4)&=-\frac{1}{2}\sum_{i=1}^4g(J\psi^{-1}[e_4,\psi e_i],e_i)=0.
\end{align*}
From this, we see that
$$
\eta_P = -\frac{1}{2}e^{4x_4}\theta^2.
$$
For comparison, we also compute the local connection 1-form $\widetilde{A}$ of $\widetilde{D}$ with respect to $f_i:=\psi e_i$, $i=1,2$:
\[
\begin{aligned}
\widetilde{A} 
= \begin{pmatrix}
    -\frac{\sqrt{-1}}{2} e^{4x_4} \theta^2 & \sqrt{-1}\xi \theta^1 + \zeta\theta^3\\
    \sqrt{-1} \xi\theta^1 - \zeta \theta^3 & 0
\end{pmatrix},
\end{aligned}
\]
where $\xi = \frac{1}{4}e^{2x_4 + 2} + e^{-2x_4 - 2} $ and $\zeta = \frac{1}{4}e^{6x_4 + 2} + e^{2x_4 -2}$.  From this, we have 
$$
\operatorname{Im}\mathrm{Tr}(\widetilde{A})=-\frac{1}{2}e^{4x_4}\theta^2
$$
which is in agreement with $\eta_P$.
\end{example}

\section{Conclusion}
\label{secConclusion}
In this note, we have established a formula for computing the Chern-Ricci form of a twisted almost K\"{a}hler manifold $(M,g^\psi,J^\psi,\omega)$ for $\psi$ a twist map.  The formula given by $\eta_P$ in Theorem \ref{thmMain} allows one to compute the Chern-Ricci form of the twisted almost K\"{a}hler structure $(g^\psi,J^\psi,\omega)$ without actually computing the Chern connection itself and thus offers a distinct computational advantage.  Future work will focus on studying the formula in depth and extracting geometric consequences from it.

\end{document}